\documentclass{article}
\usepackage{titlesec}
\usepackage{titling}
\titleformat{\chapter}{\normalfont\huge}{\thechapter.}{20pt}{\huge\it}
\renewcommand{\paragraph}{\subsection}
\usepackage[utf8]{inputenc}
\usepackage[T1]{fontenc}
\usepackage{amsmath,amssymb,amsthm,mathrsfs,calligra,amsfonts}
\usepackage{enumerate}
\usepackage[framemethod=tikz]{mdframed}
\usepackage{array}
\usepackage{lipsum}
\usepackage{mathtools}
\usepackage{scrextend}
\usepackage{tikz-cd}
\usepackage{cite}
\usepackage{pdfrender,xcolor}
\usepackage{lmodern}
\usepackage[hidelinks]{hyperref}
\usepackage[noabbrev,capitalize]{cleveref}
\theoremstyle{plain}
\newtheorem{theorem}{Theorem}[subsection]   
\newtheorem{corollary}[theorem]{Corollary}

\newtheorem{proposition}[theorem]{Proposition}

\theoremstyle{definition}
\newtheorem{definition}[theorem]{Definition}

\newtheorem{example}[theorem]{Example}

\theoremstyle{remark}
\newtheorem{remark}[theorem]{Remark}
\titleformat*{\paragraph}{\itshape\mdseries}
\titleformat*{\subsection}{\itshape\mdseries}
\crefformat{section}{\S#2#1#3}

\newcommand{\mr}{\mathrm}

\newcommand{\C}{\mathbf C}

\newcommand{\canl}{\tau^{\le p}}
\newcommand{\cang}{\tau^{\ge p}}

\newcommand{\3}{\vspace{3mm}}
\newcommand{\2}{\vspace{2mm}}

\newcommand{\Nb}{\mathscr{N}^{\vee}_\mr{Z/X}}

\newcommand*{\shom}{\mathscr{H}\text{\kern -3pt {\it{om}}}\,}
\newcommand*{\ext}{\mathscr{E}\text{\kern -1.5pt {\it{xt}}}\,}
\makeatletter
\newcommand{\xdashrightarrow}[2][]{\ext@arrow 0359\rightarrowfill@@{#1}{#2}}
\newcommand{\xdashleftarrow}[2][]{\ext@arrow 3095\leftarrowfill@@{#1}{#2}}
\newcommand{\xdashleftrightarrow}[2][]{\ext@arrow 3359\leftrightarrowfill@@{#1}{#2}}
\def\rightarrowfill@@{\arrowfill@@\relax\relbar\rightarrow}
\def\leftarrowfill@@{\arrowfill@@\leftarrow\relbar\relax}
\def\leftrightarrowfill@@{\arrowfill@@\leftarrow\relbar\rightarrow}
\def\arrowfill@@#1#2#3#4{%
  $\m@th\thickmuskip0mu\medmuskip\thickmuskip\thinmuskip\thickmuskip
   \relax#4#1
   \xleaders\hbox{$#4#2$}\hfill
   #3$%
}
\makeatother
\usepackage{geometry}               
\usepackage[parfill]{parskip}    
\usepackage{graphicx}
\usepackage{epstopdf}
\title{Degeneration of spectral sequences and complex Lagrangian submanifolds}
\author{Borislav Mladenov}
\AtEndDocument{\bigskip{\footnotesize%
  \text{Department of Mathematics, Imperial College London, London SW7 2AZ, United Kingdom} \\ 
  \href{mailto:bm1514@ic.ac.uk}{\textit{bm1514@ic.ac.uk}}
  }}
\predate{}
\date{} 
\postdate{}
\begin{document}
\maketitle
\nocite{*}

 \begin{abstract}
  There is a local-to-global $\mr{Ext}$ spectral sequence $\mr{E}_{2}^{p,q}= \mr{H}^{p}(\mr{L},\Omega^{q}_\mr{L}) \Rightarrow \mr{Ext}^{p+q}(i_{*}\mathscr{O}_\mr{L},i_{*}\mathscr{O}_{\mr{L}})$  for a smooth Lagrangian subvariety in a hyperkähler variety. We prove its degeneration on $\mr{E}_{2}$, and various generalisations thereof.
 \end{abstract}

\paragraph{Introduction.}Let $i:\mr{Z \xhookrightarrow{} X}$ be a locally complete intersection over an algebraically closed field $k$ of characteristic $0$ with normal bundle $\mathscr{N}_\mr{Z/X}$. For any $\mathscr{L} \in \mr{Pic}(\mr{Z})$, there's a local-to-global $\mr{Ext}$ spectral sequence: \begin{equation*}\mr{E}_{2}^{p,q}=\mr{H}^{p}(\mr{Z},\wedge^{q}\mathscr{N}_\mr{Z/X}) \Rightarrow \mr{Ext}^{p+q}(i_{*}\mathscr{L},i_{*}\mathscr{L}).\end{equation*} 
Notice that the left-hand side gives the de Rham cohomology of $\mr{L}$, independent of $\mathscr{L}$, but the differentials depend on $\mathscr{L}$. We are interested in the case where $i: \mr{L \xhookrightarrow{} X}$ is a smooth Lagrangian in a hyperkähler variety over $k$. In this case we have $\mathscr{N}_\mr{L/X} \cong \Omega_{L}$, the second page therefore becomes $\mr{E}_{2}^{p,q} = \mr{H}^{p}(\mr{L},\Omega^{q}_\mr{L})$. Our main results go as follows:
\2
\begin{theorem}\label{mainthm}
 Let $\mr{X}/k$ be a (projective) hyperkähler\footnote{(See \cref{sec2}, \cref{algkahler})} variety, let $i : \rm L \xhookrightarrow{} X$ be a smooth Lagrangian, denote the Kähler form on $\mr{L}$ by $\omega \in \mr{H^{1}(L,\Omega^{1}_{L})}$, and suppose $\mathscr{L}$ is a line bundle on $\mr{L}$. Then the local-to-global $\mr{Ext}$ spectral sequence  $$\mr{E}_{2}^{p,q}= \mr{H}^{p}(\mr{L},\Omega^{q}_\mr{L}) \Rightarrow \mr{Ext}^{p+q}(i_{*}\mathscr{L},i_{*}\mathscr{L})$$ degenerates on the second page if and only if $\rm d_{2}(\omega)=0$.
\end{theorem}
\2
\begin{theorem}\label{thm12}
 Let $\mr{X}/k$ be a (projective) hyperkähler variety, $i : \mr{L \xhookrightarrow{} X}$ a smooth Lagrangian, and let $\mathscr{L}$ be a line bundle on $\mr{L}$, extending to the first infinitesimal neighbourhood of $\mr{L}$ in $\mr{X}$, such as $\mathscr{O}_{\mr{L}}$. Then the local-to-global $\mr{Ext}$ spectral sequence  $$\mr{E}_{2}^{p,q}= \mr{H}^{p}(\mr{L},\Omega^{q}_\mr{L}) \Rightarrow \mr{Ext}^{p+q}(i_{*}\mathscr{L},i_{*}\mathscr{L})$$
 degenerates on the second page. Hence $\mr{H}^{*}(\mr{L}/k) = \oplus_{p,q}\mr{H}^{p}(\mr{L},\Omega^{q}_\mr{L}) = \mr{Ext}(i_{*}\mathscr{L},i_{*}\mathscr{L})$.
\end{theorem}
\2
\begin{theorem}\label{thm13}
 Let $i : \rm L \xhookrightarrow{} X$ be as above, and let $\mathscr{K}$ be any (existing) rational power of the canonical bundle of $\rm L$. Then the local-to-global $\mr{Ext}$ spectral sequence  $$\mr{E}_{2}^{p,q}= \mr{H}^{p}(\mr{L},\Omega^{q}_\mr{L}) \Rightarrow \mathrm{Ext}^{p+q}(i_{*}\mathscr{K},i_{*}\mathscr{K})$$ degenerates on the second page.
\end{theorem}

\paragraph{Method of proof.}The condition $\rm d_{2}(\omega)=0$ is equivalent to commutativity of the differentials with the Lefschetz operator. The proof of \autoref{mainthm}, given in \cref{sec2}, is an extension of some ideas of Deligne in \cite{PMIHES_1968__35__107_0}, \cite{MR0498551}, \cite{MR1265526}, which use Lefschetz-like operators to prove degeneration of spectral sequences. These are based on the following simple observation that if $\mr{d : H^{*}(X}/k) \to \mr{H^{*}(X}/k)$ is a linear map of degree $-1$, which commutes with the Lefschetz operator, then $\rm d=0$. As noted above, in our setting a local calculation shows that $\mathrm{E}_{2}^{p,q}=\mr{H}^{p}(\mr{L},\Omega^{q}_\mr{L})$, and thinking of $\mr{E}_{2}^{p,q} \subset \mr{H}^{p+q}(\mr{L}/k)$, the differentials $\mr{d}_{r}$ of our spectral sequence are of (total) degree $+1$ and, even if they commute with Lefschetz, Deligne's observation doesn't apply immediately. In the case of comutativity with the Lefschetz operator, we are, however, able to apply Deligne's idea to a part of the differential, and the outcome is that $\mr{d}_{r}$ preserves primitive cohomology and vanishes on middle primitive cohomology. Then a downward induction argument allows us to conclude that $\mr{d}_{r}=0$.\\
$\hspace*{3mm}$To prove that $\mr{d}_{2}(\omega)=0$ under the hypotheses of \cref{thm12} and \cref{thm13}, we need a second technical ingredient - it involves explicitly identifying some of the differentials on the second page. These remarks apply in the general case of a locally complete intersection $i:\mr{Z \xhookrightarrow{} X}$ and occupy most of \cref{sec1}. Here we describe the simplest and also most important case. There is an obstruction class $\alpha_{\mathscr{L}} \in \mr{Ext}^{2}(\mathscr{L},\mathscr{L}\otimes\mathscr{N}^{\vee}_{\mr{L/X}})$ to extending $\mathscr{L}$ from $\mr{Z}$ to $\mr{2Z}$.\footnote{Defined by $\mathscr{I}_\mr{Z}^{2}$} It is a morphism $$\mathscr{L} \to \mathscr{L}\otimes\mathscr{N}^{\vee}_{\mr{L/X}}[2].$$ Taking the adjoint, we get a morphism $$\mathscr{N}_\mr{Z/X} \to \mathscr{O}_\mr{Z}[2],$$ which induces our differential $$\mr{d}_{2}^{p,1} :\mr{H}^{p}(\mr{Z},\mathscr{N}_\mr{Z/X}) \to \mr{H}^{p+2}(\mr{Z},\mathscr{O}_\mr{Z}).$$ In particular, we see this vanishes if $\mathscr{L}$ lifts to the first infinitesimal neighbourhood. 

\paragraph{Generality.}We note that over $\C$ the results obtained here hold in much broader generality. It is enough to require $\mr{X}$ holomorphic symplectic (not necessarily proper!) and $\mr{L}$ compact Kähler.

\paragraph{\textbf{Context.}}
The results obtained here are motivated by and seem to mirror the degeneration of $\mr{H}^{*}(\mr{L}/\C)\otimes \Lambda \Rightarrow \mr{HF}^{*}(\mr{L,L})$ of Solomon and Verbitsky in \cite{verb}.\\
$\hspace*{3mm}$In \cite{bbdjs} Joyce et al. construct a perverse sheaf $\mathscr{P}$ on the intersection of two oriented Lagrangians in a holomorphic symplectic variety following Joyce's notion of an algebraic $\mr{d}$-critical locus. This gives a first categorification of the intersection numbers of Lagrangians. It follows by \cref{thm13m}, and a calculation in \cite{igb}, that we have:
\begin{proposition}$\mr{dim}\big(\mathrm{Ext}^{i}\big(i_{*}\mr{K_{L}^{1/2}},i_{*}\mr{K_{L}^{1/2}}\big)\big) =\mr{dim }\big(\mr{R}^{i-n}\Gamma(\mathscr{P}_\mathrm{L,L})\big)$.\end{proposition}
\2
$\hspace*{3mm}$The connection of the degeneration results obtained here with perverse sheaves and deformation quantisation is the subject of \cite{bm2}. We give new proofs, over $\C$ and under weaker hypotheses, of some of the degeneration theorems discussed here, using deformation quantisation modules. Furthermore, we prove the formality of the differential graded algebra $\mr{RHom}\big(i_*\mr{K}_\mr{L}^{1/2},i_*\mr{K}_\mr{L}^{1/2}\big)$ as well as various generalisations to pairs of Lagrangians.\\
$\hspace*{3mm}$In \cite{Kapustin_three-dimensionaltopological} Kapustin and Rozansky study the RW model which turns out to be closely related to the deformation quantisation of the $2$-periodic derived category of the target space. They conjecture the existence of a $2$-category associated to a holomorphic symplectic variety whose simplest objects are Lagrangians. In the special case of (a deformation of) the cotangent bundle, the endomorphism category of the zero section in the conjectured $2$-category is a monoidal deformation of the $2$-periodic derived category of the underlying complex manifold. Then the $\mr{Ext}$ groups could be realised as Hochschild homology of the $2$-periodic derived category which might hint at the collapse of the $\mr{Ext}$ spectral sequence.

\paragraph{Plan of paper.}
\cref{sec1} contains general results on multiplication in spectral sequences and the homological algebra of locally complete intersections. In more detail, we recall multiplicative structures on spectral sequence and review some classical homological algebra, mostly due to Deligne. After that we have a reminder on deformation-obstruction classes following Huybrechts-Thomas and Arinkin-Căldăraru. Our contributions in this sections are explicit calculations of the differentials of the local-to-global $\mr{Ext}$ spectral sequence in terms of these classes. In \cref{sec2} we prove our main results on degeneration of spectral sequences in the absolute case, i.e. over a point, while \cref{sec3} is devoted to various generalisations of the results in \cref{sec2}: we describe relative versions of \cref{mainthm}, \cref{thm12} and \cref{thm13}, and speculate on a possible generalisation to coisotropic subvarieties.

\paragraph{Acknowledgements.}I would like to thank my supervisor Richard Thomas and Daniel Huybrechts for suggesting the problem, many helpful discussions, suggestions and corrections. Thanks to Julien Grivaux, Travis Schedler and Jake Solomon for their useful comments. This work has been supported by EPSRC [EP/R513052/1], President's PhD Scholarship, Imperial College London. 

\section{General results}\label{sec1}

\paragraph{Notation.} 
We shall be working throughout over an algebraically closed field $k$ of characteristic $0$. By $\rm X$ we denote, in general, a scheme, $\rm Coh(X)$ is the abelian category of coherent sheaves on $\rm X$ and its bounded derived category is $\mr{D^{b}(X)}$. Sometimes we shall need the derived categories $\mr{{\mr{D}}^{\pm}(X)}$ of bounded above or below complexes. These are triangulated categories, so come with a shift functor $\mathscr{F} \mapsto \mathscr{F}[1]$ and exact triangles $\mathscr{F}_{1} \to \mathscr{F}_{2} \to \mathscr{F}_{3} \to \mathscr{F}_{1}[1]$. All functors we consider will be implicitly derived, except for $\shom$, $\mr{Hom}$ and $\Gamma$. For a complex $\mathscr{F}$, we let $\mathscr{H}^{i}(\mathscr{F})$ be its $i$th cohomology sheaf, so $\mathscr{H}^{i}(\mr{R}\shom(\mathscr{F},\mathscr{G}))=\ext^{i}(\mathscr{F},\mathscr{G})$, $\mr{H}^{i}(\mr{RHom}(\mathscr{F},\mathscr{G}))=\mr{Ext}^{i}(\mathscr{F},\mathscr{G})$, in particular, we get the useful identity $\mr{Ext}^{i}(\mathscr{F},\mathscr{G})=\mr{Hom}(\mathscr{F},\mathscr{G}[i])$; similarly $\mathscr{H}^{-i}((\mathscr{F}\otimes \mathscr{G}))=\mathscr{T}or_{i}(\mathscr{F},\mathscr{G})$ and so forth. We say that $\mathscr{F}$ is a perfect complex, if it is locally isomorphic, in the derived category, to a finite complex of locally free sheaves of finite rank. The category of perfect complexes is triangulated and denoted by $\mr{Perf(X)}$.\\
$\hspace*{5mm}$ Similarly, we let $\mr{FCoh(X)}$ be the abelian category of finitely filtered objects of $\mr{Coh(X)}$. We denote $\mr{D^{\pm}F(X)}$ the filtered derived categories of Deligne which are localisations with respect to the class of filtered quasi-isomorphisms. Naturally we have filtered derived functors and for a left exact $\mr{T: FCoh(X) \to Vect}$ and any $\mr{(K,F) \in D^{\pm}F(X)}$, we get a spectral sequence \begin{equation}\label{eq:1}\mr{E}_{1}^{p,q} = \mr{R}^{p+q}\mr{T(Gr}_\mr{F}^{p}(\mr{K})) \Rightarrow \mr{R}^{p+q}\mr{T(K)}.\end{equation}This construction is a functor from the filtered derived category to the category of cohomological spectral sequences of vector spaces.

\paragraph{Homological algebra.}
We begin by collecting some general results on multiplication in spectral sequences and sheaves on subvarieties.\2
\begin{definition}
Let $\mr{E}_{2}^{'p,q} \Rightarrow \mr{H}^{'p+q}$, $\mr{E}_{2}^{''p,q} \Rightarrow \mr{H}^{''p+q}$, $\mr{E}_{2}^{p,q} \Rightarrow \mr{H}^{p+q}$ be spectral sequences. A pairing of $\mr{E}_{2}^{'p,q} \Rightarrow \mr{H}^{'p+q}$ and $\mr{E}_{2}^{''p,q} \Rightarrow \mr{H}^{''p+q}$ to $\mr{E}_{2}^{p,q} \Rightarrow \mr{H}^{p+q}$ is a family of maps $$\mr{\cup_{r} : E}_{r}^{'p,q}\otimes \mr{E}_{r}^{''s,t} \to \mr{E}_{r}^{p+s,q+t} \text{ and } \mr{\cup} : \mr{H}^{'p}\otimes \mr{H}^{''q} \to \mr{H}^{p+q}$$ such that $\mr{\cup}_{r}$ are compatible with the differentials, $\cup_{r+1}$ is induced by $\mr{\cup}_{r}$, $\mr{colim} \cup_{r} = \cup_{\infty}$ and $\mr{\cup_{\infty}=gr(\cup)}$.
\end{definition}
\2
\begin{definition}
 Given a Grothendieck spectral sequence $\mr{E}_{2}^{p,q}(\mr{S})=\mr{R}^{p}\mr{TR}^{q}\mr{G(S)} \Rightarrow \mr{H}^{p+q}\mr{(S)}$ such that the functors $\mr{RT}$, $\mr{RG}$ and $\rm H$ have cup products, we shall say it has cup products if, given a pairing $\mr{S_{1}\otimes S_{2} \to S}$, we have a natural pairing of the corresponding spectral sequences and, in addition, the products on $\mr{E_{2}}$ and $\mr{H}$ are the ones induced by the functors.
\end{definition}
\2
\begin{proposition}
 Let $\mathscr{F}$ a be coherent sheaf on $\mr{X}$. Then the local-to-global $\mr{Ext}$ spectral sequence $$\mr{E}_{2}^{p,q}=\mr{H}^{p}(\mr{X},\ext^{q}(\mathscr{F},\mathscr{F})) \Rightarrow \mr{Ext}^{p+q}(\mathscr{F},\mathscr{F})$$ has cup products.
\end{proposition}
\2
\begin{remark}
 More generally, there are also pairings of the local-to-global $\mr{Ext}$ in the case of different sheaves arising from the compostion in local $\mr{R}\shom$.
\end{remark}

$\hspace*{5mm}$The local-to-global $\mr{Ext}$ spectral sequence arises, classically, from the natural isomorphism of derived functors $\mr{RHom} \cong \mr{R}\Gamma\circ \mr{R}\shom$. We shall give an alternative way of constructing it - via filtered complexes. This has the advantage of giving a geometric interpretation of the differentials as obstructions to formality of certain objects. For $(\mathscr{F},\mr{d}) \in \mr{D(X)}$, define the canonical filtrations: \[\tau^{\le p}\mathscr{F} = \begin{cases}  \mathscr{F}^{i}, & \text{for }  i < p \\  \mr{ker(d}^{p}), & \text{for } i=p \\ \rm 0, & \text{for }  i>p                                                                                                                                                                                    \end{cases},\\
 \tau^{\ge p}\mathscr{F} = \begin{cases}  0, & \text{for }  i < p \\  \mathscr{F}^{p}/\mr{im(d}^{p-1}), & \text{for }  i=p \\  \mathscr{F}^{i}, & \text{for }  i>p.                                                                                                                                                                                    \end{cases}
\]
The inclusion $\canl\mathscr{F} \xhookrightarrow{} \mathscr{F}$ induces isomorphisms on $\mathscr{H}^{i}$ for $i \le p$ and the canonical projection $\mathscr{F} \to \cang\mathscr{F}$ induces isomorphisms on $\mathscr{H}^{i}$ for $i \ge p$. There are canonical exact triangles $$\tau^{\le p-1}\mathscr{F} \to \canl\mathscr{F} \to \mathscr{H}^{p}(\mathscr{F})[-p] \to \tau^{\le p-1}\mathscr{F}[1],$$ $$\mathscr{H}^{p-1}(\mathscr{F})[-p+1] \to \tau^{\ge p-1}\mathscr{F} \to \cang\mathscr{F} \to \mathscr{H}^{p-1}(\mathscr{F})[-p+2].$$ Let's set $\tau^{[p-1,p]}\mathscr{F} = \tau^{\ge p-1}\tau^{\le p}\mathscr{F}$, so that there's an exact triangle: $$\mathscr{H}^{p-1}(\mathscr{F})[-p+1] \to \tau^{[p-1,p]}\mathscr{F} \to \mathscr{H}^{p}(\mathscr{F})[-p] \xrightarrow{\delta_{-p}(\mathscr{F})[-p]} \mathscr{H}^{p-1}(\mathscr{F})[-p+2],$$ and the elements $\delta_{-p}(\mathscr{F}) \in \mr{Ext}^{2}(\mathscr{H}^{p}(\mathscr{F}), \mathscr{H}^{p-1}(\mathscr{F}))$ realise universally the second differential $\mr{d_{2}}$ of the spectral sequence of any (filtered) derived functor applied to $\mathscr{F}$.\3
\begin{remark}\label{rmkdecalage}
 In particular, if $i : \mr{Z \xhookrightarrow{} X}$ is a locally complete intersection, $\mathscr{L} \in \mr{Pic(Z)}$, the spectral sequence \begin{equation*}\label{eq:2}\mathrm{E}_{2}^{p,q}=\mr{H}^{p}(\mr{X},\ext^{q}(i_{*}\mathscr{L},i_{*}\mathscr{L})) \Rightarrow \mr{Ext}^{p+q}(i_{*}\mathscr{L},i_{*}\mathscr{L})\end{equation*}                                                                                                                                                                                                                                                        can be realised from $i^{!}i_{*}\mathscr{L}$ with its canonical filtration. Indeed, Grothendieck-Verdier duality gives a filtered quasi-isomorphism between canonically filtered complexes \begin{equation*}\label{eq:5}\mr{R}\shom(i_{*}\mathscr{L},i_{*}\mathscr{L}) \simeq i_{*}\mr{R}\shom(\mathscr{L},i^{!}i_{*}\mathscr{L})\end{equation*} and applying filtered derived global sections functor, we get a spectral sequence using $(\ref{eq:1})$ which, by Deligne's décalage theorem (see \cite{MR0498551}), is the local-to-global $\mr{Ext}$ spectral sequence after renumbering $\mr{E}_{r}^{p,q} \mapsto \mr{E}_{r+1}^{2p+q,-p}$.
\end{remark}

$\hspace*{5mm}$ Let $i \rm : Z \xhookrightarrow{} X$ be a locally complete intersection. The normal bundle of $\mr{Z}$ in $\mr{X}$ is denoted by $\mathscr{N}_{\mr{Z/X}}$.
\2
\begin{proposition}Let $i \rm : Z \xhookrightarrow{} X$ be a locally complete intersection. 
 Suppose $c = \mr{codim(Z,X)}$, and let $\mathscr{F}$ be a coherent sheaf on $\rm Z$. Then $$\mathscr{H}^{-i}(i^{*}i_{*}\mathscr{F})\cong \begin{cases}
 \mathscr{F}\otimes \wedge^{i}\mathscr{N}^{\vee}_{\mr{Z/X}}, \, 0\le i \le c\\
 0, \text{ otherwise.}                                                                                                                                          \end{cases}
$$ 
\end{proposition}
\2
\begin{proposition}\label{prop2}
 Let $i \rm : Z \xhookrightarrow{} X$ be a locally complete intersection of codimension $c$. Let $\mathscr{F}$ and $\mathscr{G}$ be coherent sheaves on $Z$. 
 \begin{enumerate}
  \item Assume $\mathscr{F}$ locally free, then we have $\ext^{i}(i_{*}\mathscr{F},i_{*}\mathscr{G}) \cong \begin{cases} i_{*}(\wedge^{i}\mathscr{N}_{\mr{Z/X}} \otimes \mathscr{F}^{\vee}\otimes \mathscr{G}), \, 0\le i \le c\\
  0, \text{ otherwise. }
  \end{cases}$
  \item The Yoneda product coincides with the usual cup product. More precisely, let $\mathscr{F}$, $\mathscr{G}$ be locally free sheaves, $\mathscr{H}$ any coherent sheaf, then the Yoneda multiplication $$\ext^{i}(i_{*}\mathscr{G},i_{*}\mathscr{H}) \otimes \ext^{j}(i_{*}\mathscr{F},i_{*}\mathscr{G}) \to \ext^{i+j}(i_{*}\mathscr{F},i_{*}\mathscr{H})$$ corresponds under the above isomorphisms to $$i_{*}(\wedge^{i}\mathscr{N}_{\mr{Z/X}} \otimes \mathscr{G}^{\vee}\otimes \mathscr{H}) \otimes i_{*}(\wedge^{j}\mathscr{N}_{\mr{Z/X}} \otimes \mathscr{F}^{\vee}\otimes \mathscr{G}) \to i_{*}(\wedge^{i+j}\mathscr{N}_{\mr{Z/X}} \otimes \mathscr{F}^{\vee}\otimes \mathscr{H}),$$ given by exterior product and the natural map $\mathscr{G}\otimes \mathscr{G}^{\vee} \to \mathscr{O}_\mr{Z}$.
 \end{enumerate}
\end{proposition}
For any coherent $\mathscr{F}$ on $\mr{Z}$, we have a canonical exact triangle
$$\mathscr{F}\otimes \mathscr{N}^{\vee}_{\mr{Z/X}}[1] \to \tau^{\ge -1}i^{*}i_{*}\mathscr{F} \to \mathscr{F} \to \mathscr{F}\otimes\mathscr{N}^{\vee}_{\mr{Z/X}}[2].$$
\begin{definition}\label{hkr}
 The extension class of the above triangle $\alpha_{\mathscr{F}} \in \mr{Ext}^{2}(\mathscr{F},\mathscr{F}\otimes\mathscr{N}^{\vee}_{\mr{Z/X}})$ is called the deformation-obstruction class of $\mathscr{F}$. 
\end{definition}
\2
\begin{remark}
 In the special case of $\mathscr{N}^{\vee}_{\mr{Z/X}}$ the class $\alpha_{\mathscr{N}^{\vee}_{\mr{Z/X}}}$ is sometimes called the HKR class because of its close relationship to Hochschild-Kostant-Rosenberg-like theorems.
\end{remark}
$\hspace*{3mm}$Given our locally complete intersection $i:\mr{Z \xhookrightarrow{} X}$, we have the conormal exact sequence $$0 \to \mathscr{N}^{\vee}_{\mr{Z/X}} \to i^{*}\Omega_\mr{X} \to \Omega_\mr{Z} \to 0$$ whose class is called the Kodaira-Spencer class $\mr{KS\in Ext^{1}(\Omega_\mr{Z}, \mathscr{N}^{\vee}_{\mr{Z/X}})}$.\\
$\hspace*{3mm}$We would like to extend the definitions of the Atiyah class and the obstruction class $\alpha_{\mathscr{F}}$ to the derived category. Consider the diagonal embedding $\Delta : \mr{X \xhookrightarrow{} X\times X}$. Let $\Delta(\mr{X})^{(2)}$ be the second infinitesimal neighbourhood of $\Delta(\mr{X})$ in $\mr{X\times X}$, i.e. it is defined by the square of the ideal sheaf of the diagonal. The canonical exact sequence of a closed embedding becomes  $$0 \to \Delta_{*}\Omega_\mr{X} \to \mathscr{O}_{\Delta(\mr{X})^{(2)}} \to \Delta_{*}\mathscr{O}_\mr{X} \to 0$$ Its extension class is called the universal Atiyah class $\mr{At \in Ext^{1}(\Delta_{*}\mathscr{O}_\mr{X},\Delta_{*}\Omega_\mr{X})}$. Now we use Fourier-Mukai functors to evaluate these universal classes at particular objects in the derived category: the Atiyah class of an object $\mathscr{F} \in \mr{D^{b}(X)}$ is then $\mr{At(\mathscr{F}) = \Phi_{At}(\mathscr{F})}$, where $$\Phi_\mr{At}(\mathscr{F}) : \Phi_{\Delta_{*}\mathscr{O}_\mr{X}}(\mathscr{F}) \to \Phi_{\Delta_{*}\Omega_\mr{X}}(\mathscr{F}).$$
To define the universal obstruction class, consider the closed embedding $\tilde{i} = \mr{id}\times i : \mr{Z\times Z \xhookrightarrow{} Z\times X}$, so that the complex $\tilde{i}^{*}\tilde{i}_{*}\Delta_{*}\mathscr{O}_{\mr{Z}}$ gives an $\mr{Ext}^{2}$ class $$\alpha_{\Delta_{*}\mathscr{O}_\mr{Z}} : \Delta_{*}\mathscr{O}_\mr{Z} \to \Delta_{*}\mathscr{N}^{\vee}_\mr{Z/X}[2].$$ As above we can evaluate it at any object in the derived category $\mr{D^{b}(X)}$.
\2
\begin{proposition}(\!\cite{ARINKIN2012815}, \cite{MR3158008})
 Suppose $i:\mr{Z \xhookrightarrow{} X}$ is a locally complete intersection, and consider $\mathscr{F} \in \mr{D^{b}(X)}$. The class $\alpha_{\mathscr{F}}$ is the product of $\mr{At(\mathscr{F})}$ and $\mr{KS}$, i.e. $\alpha_{\mathscr{F}} = \mr{(id_{\mathscr{F}}\otimes KS) \circ At(\mathscr{F})}$.
\end{proposition}
The next theorem gives a geometric interpretation of the obstruction class defined above (and justifies the terminology!).
\2
\begin{theorem}(Huybrechts-Thomas \cite{MR3158008}, Grivaux \cite{2015arXiv150504414G})\label{thm:th}
 Let $j:\mr{X \xhookrightarrow{} X^{(1)}}$ be a first order thickeing of a Noetherian separated scheme. Suppose $\mathscr{F} \in \mr{D^{b}(X)}$ is a perfect complex. Then $\mathscr{F}$ extends to a perfect complex $\mathscr{F}^{(1)}$ on $\mr{X^{(1)}}$ iff $\alpha_{\mathscr{F}}=0$.
\end{theorem}\3
\begin{theorem}(Arinkin-Căldăraru \cite{ARINKIN2012815})\label{thm:th2}
 Let $i:\mr{Z} \xhookrightarrow{} \mr{X}$ be a locally complete intersection. Suppose $\mathscr{F}$ is locally free, $\mathscr{F}_{0}$ - quasi-coherent on $\mr{Z}$. Consider a morphism $\mr{m: \mathscr{F}\otimes \Nb \to \mathscr{F}_{0}}$. Then the obstruction to existence of an exact sequence $$0 \to i_{*}\mathscr{F}_{0} \to \mathscr{G} \to i_{*}\mathscr{F} \to 0,$$ such that the ideal sheaf of $\mr{Z}$ acts on $\mathscr{G}$ via $\mr{m}$ is $\mr{m} \circ \alpha_{\mathscr{F}}$. 
\end{theorem}
$\hspace*{5mm}$We are ready explain the relationship between these obstruction classes and the differentials on the second page of the spectral sequence. Thinking of $\alpha_{\mathscr{L}}$ as an extension class, it is clear that $\delta_{0}(i^{*}i_{*}\mathscr{L})=\alpha_{\mathscr{L}}$ as $i^{*}i_{*}\mathscr{L}$ is concentrated in non-positive degrees. For the other $\mr{Ext}$ classes, we have to work harder. Let us factor $i : \mr{Z\xhookrightarrow{} X}$ as follows: $$ \mr{Z} \xhookrightarrow{j}
\mr{Z}^{(1)} \xhookrightarrow{i^{(1)}} \mr{X},$$ where as usual $\mr{Z}^{(1)}$ is the first infinitesimal neighbourhood of $\mr{Z}$ in $\mr{X}$. 
In the case of $\mathscr{O}_\mr{Z}$, one can calculate $\delta_{1}(i^{*}i_{*}\mathscr{O}_\mr{Z})$ in terms of the deformation-obstruction classes:\2
\begin{proposition}(Arinkin-Căldăraru)
 Let $i : \mr{Z \xhookrightarrow{} X}$ be a locally complete intersection. Then we have \begin{equation*}\delta_{1}(i^{*}i_{*}\mathscr{O}_\mr{Z}) = \alpha_{\Nb}.\end{equation*}
\end{proposition}
\begin{proof}
Consider the exact sequence 
\begin{equation}\label{eq:3} 0 \to j_{*}\Nb \to \mathscr{O}_\mr{Z^{(1)}} \to j_{*}\mathscr{O}_\mr{Z}\to 0 \end{equation} 
apply $j^{*}$ to get an exact triangle $$j^{*}j_{*}\Nb \to \mathscr{O}_\mr{Z} \to j^{*}j_{*}\mathscr{O}_\mr{Z} \to j^{*}j_{*}\Nb[1],$$ hence an isomorphism \begin{equation}\label{eq:10}\tau^{<0}j^{*}j_{*}\mathscr{O}_\mr{Z} \cong j^{*}j_{*}\Nb[1]\end{equation} in $\mr{D^{b}(X)}$. This implies that $$\delta_{1}(j^{*}j_{*}\mathscr{O}_\mr{Z}) = \delta_{1}(j^{*}j_{*}\Nb[1]) = \delta_{0}(j^{*}j_{*}\Nb) = \alpha_{\Nb}.$$
We let $\mr{Alt} : \mathscr{F}^{\otimes p} \to \bigwedge^{p} \mathscr{F}$ be the natural quotient. Consider an exact triangle $$\mr{Sym}^{2}\Nb[2] \to \tau^{\ge -2}j^{*}j_{*}\mathscr{O}_\mr{Z} \to \mathscr{C} \to \mr{Sym}^{2}\Nb[3],$$ where the first map is the inclusion $$\mr{Sym^{2}\Nb \to \Nb\otimes \Nb}$$ followed by the canonical morphisms $$\Nb\otimes \Nb \to \mathscr{H}^{-2}(j^{*}j_{*}\mathscr{O}_\mr{Z}) \text{ and } \mathscr{H}^{-2}(j^{*}j_{*}\mathscr{O}_\mr{Z}) \to \tau^{\ge- 2}j^{*}j_{*}\mathscr{O}_\mr{Z}.$$ By construction we have $\delta_{1}(\mathscr{C}) = \mr{Alt} \circ \delta_{1}(j^{*}j_{*}\mathscr{O}_\mr{Z})$, and since $$\tau^{\ge-2}i^{*}i_{*}\mathscr{O}_\mr{Z} \to \tau^{\ge -2}j^{*}j_{*}\mathscr{O}_\mr{Z} \to \mathscr{C}$$ is an isomorphism in $\mr{D^{b}(X)}$, we conclude $$\delta_{1}(i^{*}i_{*}\mathscr{O}_\mr{Z}) = \delta_{1}(\mathscr{C}) = \mr{Alt} \circ \delta_{1}(j^{*}j_{*}\mathscr{O}_\mr{Z})= \mr{Alt}\circ \alpha_{\Nb}.$$ 
$\autoref{thm:th2}$, and the exact sequence 
$$ 0 \to i_{*}\mr{Sym}^{2}\Nb=\mathscr{I}^{2}_\mr{Z}/ \mathscr{I}^{3}_\mr{Z} \to \mathscr{I}_\mr{Z}/ \mathscr{I}^{3}_\mr{Z} \to i_{*}\Nb \to 0$$ imply that $$\mr{Alt}\circ \alpha_{\Nb} =\alpha_{\Nb},$$ since $\mr{Sym}\circ\alpha_{\Nb}=0$, that is, the obstruction class of the conormal bundle is skew-symmetric and we are done.
\end{proof}
\2
\begin{remark}
It would be nice to have similar geometric interpretation for the higher $\delta$'s in the general case.
\end{remark}
\2
\begin{remark}
 In general, if $\mathscr{F} \in \mr{Perf}(\mr{Z})$ lifts to the first infinitesimal neighbourhood, we have $\delta_{1}(i^{*}i_{*}\mathscr{F}) = \alpha_{\mathscr{F}\otimes \Nb}$. We note that Arinkin and Căldăraru in \cite{ARINKIN2012815} show that $i^{*}i_{*}\mathscr{F}$ is formal iff $\alpha_{\mathscr{F}}$ and $\alpha_{\mathscr{F}\otimes \Nb}$ vanish. In the case of structure sheaf, formality is understood in the (stronger) sense of differential graded algebras.
\end{remark}
\2
\begin{proposition}
 Let $i : \mr{Z \xhookrightarrow{} X}$ be a locally complete intersection, $c=\mr{codim(Z,X)}$, and consider $\mathscr{L} \in \mr{Pic(Z)}$. Define $$\tilde\delta_{q}(\mathscr{L}) \coloneqq \delta_{c-q}(i^{*}i_{*}\mathscr{L})\otimes \mr{id_{\mathscr{L}^{\vee}}\otimes id_{det\mathscr{N}_\mr{Z/X}}}$$ The differential $\mr{d_{2}}$ of $$\mr{E}_{2}^{p,q}=\mr{H}^{p}(\mr{X},\ext^{q}(i_{*}\mathscr{L},i_{*}\mathscr{L})) \Rightarrow \mr{Ext}^{p+q}(i_{*}\mathscr{L},i_{*}\mathscr{L})$$ can be described as \begin{equation*} \mr{d}_{2}^{p,q} = \mr{R}^{p}\Gamma \tilde\delta_{q}(\mathscr{L}) : \mr{H}^{p}(\mr{Z},\wedge^{q}\mathscr{N}_\mr{Z/X}) \to \mr{H}^{p+2}(\mr{Z},\wedge^{q-1}\mathscr{N}_\mr{Z/X}).\end{equation*}
\end{proposition}
\begin{proof}
 As already explained Grothendieck-Verdier duality gives an isomorphism in $\mr{D^{b}F(X)}$ between canonically filtered complexes $$\mr{R}\shom(i_{*}\mathscr{L},i_{*}\mathscr{L}) \simeq i_{*}\mr{R}\shom(\mathscr{L},i^{!}i_{*}\mathscr{L}).$$ By definition (see \cref{rmkdecalage}), $$ \mr{d}_{2}^{p,q} = \mr{R}^{p}\Gamma\delta_{-q}(\mr{R}\shom(i_{*}\mathscr{L},i_{*}\mathscr{L})),$$ hence we are done since there is a canonical isomorphism  $$i_{*}\mr{R}\shom(\mathscr{L},i^{!}i_{*}\mathscr{L}) \simeq i_{*}(\mathscr{L}^{\vee}\otimes i^{!}i_{*}\mathscr{L})$$ in $\mr{D^{b}F(X)}$.
\end{proof}
\begin{proposition}\label{lemmavanishing}
 Let $i : \mr{Z \xhookrightarrow{} X}$ be a locally complete intersection and consider any $\mathscr{F} \in \mr{Perf}(\mr{Z})$. Then $i_{*}\alpha_{\mathscr{F}}=0$. Furthermore, the morphism $$\mr{Ext}^{2}(\mathscr{F},\mathscr{F}\otimes \Nb) \to \mr{Ext}^{2}(i_{*}\mathscr{F},i_{*}(\mathscr{F}\otimes \Nb))$$ is injective iff $\alpha_{\mathscr{F}}=0$.
\end{proposition}
\2
\begin{remark}\label{obviouscase}
 We observe that there are two obvious cases $\mathscr{L} = \mathscr{O}_\mr{Z}$ and $\mathscr{L} = \mr{det}\mathscr{N}_\mr{Z/X}$ in which the differential $\mr{d}_{2}^{p,c} : \mr{H}^{p}(\mr{Z},\wedge^{c}\mathscr{N}_\mr{Z/X}) \to \mr{H}^{p+2}(\mr{Z},\wedge^{c-1}\mathscr{N}_\mr{Z/X})$ vanishes. For the structure sheaf, it's enough to note that $\delta_{0}(i^{*}i_{*}\mathscr{O}_\mr{Z})=0$, while for $\mr{det}\mathscr{N}_\mr{Z/X}$, we use \cref{lemmavanishing}.
\end{remark}

\paragraph{Traces.}
We briefly review traces, mainly to fix notation. Suppose $\mathscr{F} \in \mr{Perf(X)}$ and let $\mathscr{V}$ be a vector bundle on $\mr{X}$. Then we define the sheaf trace $$\mr{Tr} : \mr{R}\shom(\mathscr{F},\mathscr{F}\otimes \mathscr{V}) \to \mathscr{V}$$ as the compostion of the natural isomorphism $\mr{R}\shom(\mathscr{F},\mathscr{F}\otimes \mathscr{V}) \simeq \mathscr{F}^{\vee}\otimes \mathscr{F}\otimes \mathscr{V}$ and the natural map $\mathscr{F}^{\vee}\otimes\mathscr{F}\to \mathscr{O}_\mr{X}$. Applying $\mr{R\Gamma}$, and taking cohomology, we get the cohomological trace map of Illusie (see \cite{MR0491680}): $$\mr{Tr} : \mr{Ext}^{k}(\mathscr{F},\mathscr{F}\otimes\mathscr{V}) \to \mr{H}^{k}(\mr{X},\mathscr{V}).$$ If $\mathscr{F}, \mathscr{G} \in \mr{Perf(X)}$, we also have a partial trace: $$\mr{Tr}_{\mathscr{F}} : \mr{Ext}^{k}(\mathscr{F}\otimes \mathscr{G},\mathscr{F}\otimes \mathscr{G}\otimes\mathscr{V}) \to \mr{Ext}^{k}(\mathscr{G},\mathscr{G}\otimes\mathscr{V}).$$   is additive in the following sense: let $\mathscr{F_{1}}\to \mathscr{F}_{2} \to \mathscr{F}_{3} \to \mathscr{F}_{1}[1]$ be an exact triangle, and suppose $\alpha_{i} \in \mr{Ext}^{k}(\mathscr{F}_{i},\mathscr{F}_{i}\otimes\mathscr{V})$ induce a morphism of exact triangles (notice $\mathscr{V}$ is flat so tensoring with $\mathscr{V}$ is exact), then $\mr{Tr\alpha_{1} -Tr\alpha_{2} + Tr\alpha_{3}}=0$. It is also multiplicative: if $\alpha \in \mr{Ext}^{k}(\mathscr{F},\mathscr{F}\otimes\mathscr{V})$ and $\beta \in \mr{Ext}^{j}(\mathscr{V},\mathscr{E})$, for a vector bundle $\mathscr{E}$, then $\mr{Tr((id}\otimes \beta)\circ \alpha) =\beta \circ \mr{Tr\alpha}$ in $\mr{H}^{i+j}(\mr{X},\mathscr{E})$. Applying this to the case of an obstruction class $\alpha_{\mathscr{V}}$, we get \begin{equation}\label{eq:4}\mr{Tr\alpha_{\mathscr{V}} = Tr((id\otimes KS)\circ At\mathscr{V}) = KS \circ Tr(At\mathscr{V}) = \alpha_{det\mathscr{V}}}.\end{equation} 

\section{Applications}\label{sec2}

\paragraph{The absolute case.}We begin with the definition of a hyperkähler variety in the algebraic setting. Afterwards, we consider the compatibility of the local-to-global spectral sequence with Serre duality and give a brief reminder on Lefschetz structures, before going into our main results.\2
\begin{definition}\label{algkahler}
 A smooth, proper, connected scheme $\mr{X}/k$ is a hyperkähler variety (over $k$) if $\mr{dim}(\mr{H}^{0}(\mr{X},\Omega^{2}_\mr{X}))=1$, generated by a non-degenerate form $\sigma$ (a symplectic form), $\pi^{\text{ét}}(\mr{X})=1$, and $\mr{K_{X}} \cong \mathscr{O}_\mr{X}$.
\end{definition}

$\hspace*{5mm}$A subvariety $\mr{L}$ of $\mr{X}$ is called Lagrangian if $\left.\sigma\right|_{\mr{L}}=0$ and $\mr{2dimL = dimX}$. If $i:\mr{ L \xhookrightarrow{} X}$ is a smooth Lagrangian we have $\mathscr{T}_\mr{X} \cong \Omega^{1}_\mr{X}$ via the symplectic form, hence $i^{*}\mathscr{T}_\mr{X} \cong i^{*}\mr{\Omega^{1}_{X}}$. There is a commutative diagram: 
\[
\begin{tikzcd}
{}
0\arrow{r}
&\mathscr{T}_\mr{L}\arrow{r}\arrow{d}
&i^{*}\mathscr{T}_\mr{X}\arrow{r}\arrow{d}
&\mathscr{N}_\mr{L/X}\arrow{r}\arrow{d}
&0\\
0\arrow{r}
&\mathscr{N}^{\vee}_\mr{L/X}\arrow{r}
&i^{*}\mr{\Omega^{1}_{X}}\arrow{r}
&\mr{\Omega^{1}_{L}}\arrow{r}
&0
\end{tikzcd}
\]
which shows we have isomorphisms $\Omega^{q}_\mr{L} \cong \wedge^{q}\mr{\mathscr{N}_{L/X}}$. Hence the second page in the Lagrangian case is $\mr{E}_{2}^{p,q} = \mr{H}^{p}(\mr{L},\Omega^{q}_\mr{L})$ and, using the results of previous sections, we have differentials $$\mr{R}^{p}\Gamma\tilde\delta_{1}(\mathscr{L})=\mr{d}_{2}^{p,1} \mr{:H}^{p}(\mr{L},\Omega^{1}_\mr{L}) \to \mr{H}^{p+2}(\mr{L},\mathscr{O}_\mr{L}),$$  $$\mr{R}^{p}\Gamma \tilde\delta_{2}(\mathscr{L})=\mr{d}_{2}^{p,2} :\mr{H}^{p}(\mr{L},\Omega^{2}_\mr{L}) \to \mr{H}^{p+2}(\mr{L},\Omega^{1}_\mr{L}).$$ 
Serre duality asserts that the pairing $$\mr{H}^{p}(\mr{L},\Omega^{q}_\mr{L}) \otimes \mr{H}^{n-p}(\mr{L},\Omega^{n-q}_\mr{L}) \to \mr{H}^{n}(\mr{L},\mr{K}_\mr{L})$$ is perfect, where we set $n=\mr{dim L = codim(L,X)}$ and $\mr{K}_\mr{L}$ denotes the canonical bundle of $\mr{L}$. Since $\mr{d_{2}}$ kills $\mr{H}^{n}\mr{(L,K_{L})}$, and the higher differentials kill the appropriate subquotients, and $\mr{d}_{2}$ is multiplicative, we see that the diagram
\[ 
 \begin{tikzcd}
 \mr{H}^{p}(\mr{L},\Omega^{q}_\mr{L})\otimes \mr{H}^{n-p}(\mr{L},\Omega^{n-q}_\mr{L})\arrow[yshift=.85ex, rr, "\mr{d}_{2}\otimes \mr{id}"] \arrow[dr, ""]
 && \mr{H}^{p+2}(\mr{L},\Omega^{q-1}_\mr{L})\otimes \mr{H}^{n-p-2}(\mr{L},\Omega^{n-q+1}_\mr{L})\arrow[yshift=-.45ex, ll,"\mr{id}\otimes \mr{d}_{2}"] \arrow[dl, ""] \\
  &\C
 \end{tikzcd}
 \]
commutes up to sign, hence the following diagram, and its variants for the higher differentials, commutes up to sign, i.e. the differentials are compatible with Serre duality:
 \[ 
 \begin{tikzcd}
 \mr{H}^{p}(\mr{L},\Omega^{q}_\mr{L})\arrow[r, "\mr{d_{2}}"] \arrow[d, ""]
 & \mr{H}^{p+2}(\mr{L},\Omega^{q-1}_\mr{L}) \arrow[d, ""] \\
 \mr{H}^{n-p}(\mr{L},\Omega^{n-q}_\mr{L})^{\vee} \arrow[r, "\mr{d_{2}^{\vee}}"]
 & \mr{H}^{n-p-2}(\mr{L},\Omega^{n-q+1}_\mr{L})^{\vee}.
 \end{tikzcd}
 \]
 \begin{definition}
 A graded Lefschetz structure, in an abelian category $\mathscr{A}$, is a pair $(\mr{H},\mr{L})$, where $\mr{H} \in \mathscr{A}$ is a graded object and $\mr{L}$ is a nilpotent endomorphism of $\mr{H}$, such that $\mr{L}(\mr{H}^{i}) \subset \mr{H}^{i+2}$ and $\mr{L}^{i} : \mr{H}^{-i} \xrightarrow{\sim} \mr{H}^{i}$.
 \end{definition}\2
 \begin{example}
  Let $\mr{X}$ be a smooth projective variety of dimension $n$ with an ample line bundle $\mathscr{L}$. Then, if $\omega$ is the first Chern class of $\mathscr{L}$ and $\mr{L}_{\omega} \coloneqq \omega\, \cup : \mr{H}^{n-*}(X/k) \to \mr{H}^{n-*}(X/k)$, the pair $(\mr{H}^{n-*}(X/k),\mr{L}_{\omega})$ is a graded Lefschetz structure (in fact, it is a $(-)$Hodge-Lefschetz structure - an example of a mixed Hodge structure).
 \end{example}
\2
\begin{remark}
 We shall abuse terminology, calling the first Chern class of an ample line bundle a Kähler form.
\end{remark}
\2
\begin{theorem}\label{thm}
  Let $\mr{X}/k$ be a (projective) hyperkähler variety, let $i : \rm L \xhookrightarrow{} X$ be a smooth Lagrangian, denote the Kähler form on $\mr{L}$ by $\omega \in \mr{H^{1}(L,\Omega^{1}_{L})}$, and suppose $\mathscr{L}$ is a line bundle on $\mr{L}$. Then the local-to-global $\mr{Ext}$ spectral sequence  $$\mr{E}_{2}^{p,q}= \mr{H}^{p}(\mr{L},\Omega^{q}_\mr{L}) \Rightarrow \mr{Ext}^{p+q}(i_{*}\mathscr{L},i_{*}\mathscr{L})$$ degenerates (multiplicatively) on the second page if and only if $\rm d_{2}(\omega)=0$.
\end{theorem}
\begin{proof}
 The only if part is trivial. Suppose $\mr{d_{2}(\omega)=0}$. Then for all $\mr{r\ge 2}$, $\mr{d}_{r}(\omega)=0$ and we get 
 $$\mr{d}_{r}(\mr{L}_{\omega}\alpha) = \mr{d}_{r}(\omega)\cup \alpha + \omega \cup \mr{d}_{r}(\alpha) = \mr{L}_{\omega}\mr{d}_{r}(\alpha),$$ where $\mr{L}_{\omega}$ is the Lefschetz operator associated to $\omega$, i.e. the differentials commute with the Lefschetz operator. Let $\mr{codim(L,X)}=n$, and consider the diagram:
  \[ 
 \begin{tikzcd}
 \mr{H}^{p}(\mr{L},\Omega^{q}_\mr{L})\arrow[r, "\mr{d_{2}}"] \arrow[d, "\mr{L}_{\omega}^{n-p-q+1}"]
 & \mr{H}^{p+2}(\mr{L},\Omega^{q-1}_\mr{L}) \arrow[d, "\mr{L}_{\omega}^{n-p-q+1}"] \\
 \mr{H}^{n-q+1}(\mr{L},\Omega^{n-p+1}_\mr{L}) \arrow[r, "\mr{d_{2}}"]
 & \mr{H}^{n-q+3}(\mr{L},\Omega^{n-p}_\mr{L}).
 \end{tikzcd}
 \]
 Restricting to the primitive cohomology $\mr{H}_{0}^{p}(\mr{L},\Omega_\mr{L}^{q})$, we see that $$\mr{d}_{2}(\mr{H}_{0}^{p}(\mr{L},\Omega_\mr{L}^{q})) \subset \mr{ker}(\left.\mr{L}^{n-p-q+1}_{\omega}\right|_{\mr{H}^{p+2}(\mr{L},\Omega_\mr{L}^{q-1})}) = \mr{H}_{0}^{p+2}(\mr{L},\Omega_\mr{L}^{q-1})\oplus \mr{L}_{\omega}\mr{H}_{0}^{p+1}(\mr{L},\Omega_\mr{L}^{q-2}).$$ So we can write $\mr{d_{2} = d_{2}^{o} + L_{\omega}d_{2}^{'}}$. Deligne's argument is essentially the diagram:
  \[ 
 \begin{tikzcd}
 \mr{H}_{0}^{p}(\mr{L},\Omega^{q}_\mr{L})\arrow[r, "\mr{d^{'}_{2}}"] \arrow[d, "\mr{L}_{\omega}^{n-p-q+1}"]
 & \mr{H}^{p+1}(\mr{L},\Omega^{q-2}_\mr{L}) \arrow[d, "\mr{L}_{\omega}^{n-p-q+1}"] \\
 \mr{H}^{n-q+1}(\mr{L},\Omega^{n-p+1}_\mr{L}) \arrow[r, "\mr{d^{'}_{2}}"]
 & \mr{H}^{n-q+2}(\mr{L},\Omega^{n-p-1}_\mr{L}).
 \end{tikzcd}
 \]
 The left vertical arrow is $0$, while the right one is injective, hence $\mr{d_{2}^{'}}=0$. This means that $\mr{d_{2}}$ preserves primitive cohomology and $\mr{d}_{2}^{p,q}$ vanishes for all $p+q=n$. Let us assume that $\rm d_{2}$ vanishes for $ p+q=k+1$, base case being $p+q=n$. Suppose $\alpha \in \mr{H}_{0}^{p}(\mr{L},\Omega^{q}_\mr{L})$ with $p+q=k$. Then for any $\beta \in \mr{H}_{0}^{q-1}(\mr{L},\Omega^{p+2}_\mr{L})$, we have $\mr{L}^{n-p-q-1}_{\omega}(\alpha \cup \beta)=0$, so 
 \begin{equation*}
 \begin{split}
   0=\mr{d}_{2}(\mr{L}^{n-p-q-1}_{\omega}(\alpha \cup \beta)) &= \mr{L}^{n-p-q-1}_{\omega}(\mr{d}_{2}(\alpha) \cup \beta + (-1)^{\mr{deg}\alpha}\alpha\cup \mr{d}_{2}(\beta))\\
   &= \mr{L}^{n-p-q-1}_{\omega}(\mr{d}_{2}(\alpha) \cup \beta).
 \end{split}
 \end{equation*}
We note that the pairing $$\mr{H}_{0}^{p+2}(\mr{L},\Omega^{q-1}_\mr{L}) \otimes \mr{H}_{0}^{q-1}(\mr{L},\Omega^{p+2}_\mr{L}) \to k : \gamma \otimes \beta \mapsto \int_{\mr{L}} \omega^{n-p-q-1}\cup \gamma\cup \beta$$ is non-degenerate. Therefore $\mr{d_{2}(\alpha)}=0$ and by induction we conclude that $\mr{d_{2}=0}$. Assuming by induction $\mr{d}_{r}=0$, we run the same procedure to $\mr{d}_{r+1}$ to complete the induction, hence concluding that for all $r\ge 2$, $\mr{d}_{r}=0$, and the spectral sequence collapses on the second page.
\end{proof}
\begin{theorem}\label{cor}
 Let $\mr{X}/k$ be a (projective) hyperkähler variety, $i : \mr{L \xhookrightarrow{} X}$ a smooth Lagrangian, and let $\mathscr{L}$ be a line bundle on $\mr{L}$, extending to the first infinitesimal neighbourhood of $\mr{L}$ in $\mr{X}$, such as $\mathscr{O}_{\mr{L}}$. Then the local-to-global $\mr{Ext}$ spectral sequence  $$\mr{E}_{2}^{p,q}= \mr{H}^{p}(\mr{L},\Omega^{q}_\mr{L}) \Rightarrow \mr{Ext}^{p+q}(i_{*}\mathscr{L},i_{*}\mathscr{L})$$
 degenerates on the second page. Hence $\mr{H}^{*}(\mr{L}/k) = \oplus_{p,q}\mr{H}^{p}(\mr{L},\Omega^{q}_\mr{L}) = \mr{Ext}(i_{*}\mathscr{L},i_{*}\mathscr{L})$.
\end{theorem}
\begin{proof}
 It would suffice to show that $\mr{d_{2}}(\omega)=0$. By Serre duality, $\mr{d}_{2}^{p,1}=0$ iff $\mr{d}_{2}^{n-p-2,n}=0$.
 We have already seen that $\mr{d}_{2}^{p,n}=\mr{R}^{p}\Gamma(\alpha_{\mathscr{L}}\otimes \mr{id_{\mathscr{L}^{\vee}}\otimes id_{K_{L}}})$. The class $\alpha_{\mathscr{L}}$ is the obstruction to extending $\mathscr{L}$ to a line bundle on the first infinitesimal neighbourhood, so vanishes iff $\mathscr{L}$ extends and hence $\mr{d_{2}}(\omega)=0$.
\end{proof}
\begin{theorem}\label{thm13m}
 Let $i : \rm L \xhookrightarrow{} X$ be as above and let $\mathscr{K}$ be any (existing) rational power of the canonical bundle of $\rm L$. Then the local-to-global $\mr{Ext}$ spectral sequence  $$\mr{E}_{2}^{p,q}= \mr{H}^{p}(\mr{L},\Omega^{q}_\mr{L}) \Rightarrow \mathrm{Ext}^{p+q}(i_{*}\mathscr{K},i_{*}\mathscr{K})$$ degenerates (multiplicatively) on the second page.
\end{theorem}
\begin{proof}
 Let $\mathscr{K} = \mr{K_{L}}^{s/t}$ and note that since $\mathscr{K}$ is a line bundle $$\alpha_{\mathscr{K}} = \mr{KS} \cup \mr{c_{1}}(\mathscr{K}) \in \mr{H}^{2}\mr{(L,\mathscr{T}_{L})},$$ hence $\alpha_{\mathscr{K}} = (s/t)\alpha_{\mr{K_{L}}}$ in $\mr{H^{2}(L,\mathscr{T}_{L})}$, and then $\tilde\delta_{n}(\mathscr{K})= (s/t)\tilde\delta_{n}(\mr{K_{L}})$, so the proof below shows we may suppose that $\mathscr{K}=\mr{K}_{\mr{L}}$ is the canonical bundle of $\mr{L}$. We would like to show that $\mr{d_{2}(\omega)=0}$ and apply the theorem. Notice, as in the proof of \cref{cor}, Serre duality implies that it suffices to show that $\mr{d}_{2}^{p,n}=0$. However, in the present case, we have \begin{equation*}\mr{d}_{2}^{p,n} =\mr{R}^{p}\Gamma(\alpha_{\mr{K}_{L}}\otimes \mr{id_{\mr{K}_\mr{L}^{\vee}}\otimes id_{K_{L}}}) =\mr{R}^{p}\Gamma(\alpha_{\mr{K}_{L}}) =0,\end{equation*} where the last equality follows from \cref{obviouscase}.
 \end{proof}
 \begin{corollary}
  Let $\mr{X}/k$ be a (projective) hyperkähler variety and let $i : \rm L \xhookrightarrow{} X$ be a smooth Lagrangian. Suppose that $\mr{K}_\mr{L}^{1/2}$ is a square root of the canonical bundle of $\mr{L}$. Then the local-to-global $\mr{Ext}$ spectral sequence  $$\mr{E}_{2}^{p,q}= \mr{H}^{p}(\mr{L},\Omega^{q}_\mr{L}) \Rightarrow \mathrm{Ext}^{p+q}(i_{*}\mr{K}_\mr{L}^{1/2},i_{*}\mr{K}_\mr{L}^{1/2})$$ degenerates (multiplicatively) on the second page.
 \end{corollary}

 \section{Variants}\label{sec3}

\paragraph{The relative case.}

We briefly outline a relative version of the results obtained in the previous section.\\
$\hspace*{3mm}$It is evident that \cref{algkahler} generalises to any base scheme $\mr{S}$ of characteristic $0$, so we have a notion of hyperkähler schemes $\mr{X/S}$. Consider the following situation:
\[
\begin{tikzcd}
\mr{L} \arrow[r, hook, "i"] 
& \mr{X} \arrow[d, "p"]\\
& \mr{S},
\end{tikzcd}
\]
where $\mr{L}$ is a smooth Lagrangian in $\mr{X}$, and set  $p'= p\circ i$. Just as in the absolute case, we obtain $\mathscr{N}_\mr{L/X} \cong \Omega^{1}_\mr{L/S}$. If $\mr{X}$ is projetive, we have a relatively ample line bundle on $\mr{X}$, inducing a relative Kähler form $\omega \in \mr{H}^{1}(\mr{L},\Omega^{1}_\mr{L/S})$. We get Lefschetz operators $$\mr{R}p_{*}\omega : \mr{R}p'_{*}\Omega^{q}_\mr{L/S} \to \mr{R}p'_{*}\Omega^{q+1}_\mr{L/S}[1],$$ satisfying the hard Lefschetz theorem. Furthermore, $\mr{R}p'_{*}\Omega^{q}_\mr{L/S}$ are formal complexes, and their cohomology sheaves are locally free, hence Serre duality holds (see \cite{PMIHES_1968__35__107_0}). Therefore, we obtain the following generalisation of the absolute case considered in the previous paragraph:\3
\begin{theorem}
 Let $\mr{X/S}$ be a projecive hyperkähler variety, where $\mr{S}$ is of characteristic $0$, and let 
 \[
\begin{tikzcd}
\mr{L} \arrow[r, hook, "i"] 
& \mr{X} \arrow[d, "p"]\\
& \mr{S}.
\end{tikzcd}
\] 
be a smooth Lagrangian. Let $\omega: \mathscr{O}_\mr{L} \to \Omega_\mr{L/S}[1]$ be the relative Kähler form induced by a relatively ample line bundle and suppose $\mathscr{L} \in \mr{Pic(L)}$. Then the spectral sequence $$\mr{E}_{2}^{p,q} = \mr{R}^{p}p_{*}i_{*}\Omega^{q}_\mr{L/S} \Rightarrow \mr{R}^{p+q}p_{*}\mr{R}\shom(i_{*}\mathscr{L},i_{*}\mathscr{L})$$ degenerates (on $\mr{E}_{2}$) if and only if $\mr{d}_{2}\circ \mr{R}p_{*}\omega=0$, acting on $\mr{R}p'_{*}\mathscr{O}_\mr{L}$.
\end{theorem}
\2
\begin{theorem}
 Let $\mr{X/S}$ be a projecive hyperkähler variety, $\mr{S}$ is of characteristic $0$ as above, and let 
 \[
\begin{tikzcd}
\mr{L} \arrow[r, hook, "i"] 
& \mr{X} \arrow[d, "p"]\\
& \mr{S}.
\end{tikzcd}
\]
be a smooth Lagrangian. Let $\omega: \mathscr{O}_\mr{L} \to \Omega_\mr{L/S}[1]$ be the relative Kähler form induced by a relatively ample line bundle and suppose $\mathscr{L} \in \mr{Pic(L)}$ extends to the first infinitesimal neighbourhood of $\mr{L}$ in $\mr{X}$, such as $\mathscr{O}_\mr{L}$. Then the spectral sequence $$\mr{E}_{2}^{p,q} = \mr{R}^{p}p_{*}i_{*}\Omega^{q}_\mr{L/S} \Rightarrow \mr{R}^{p+q}p_{*}\mr{R}\shom(i_{*}\mathscr{L},i_{*}\mathscr{L})$$ degenerates on $\mr{E}_{2}$. 
\end{theorem}
\2
\begin{theorem}
 Let $\mr{X/S}$ be a projecive hyperkähler variety, $\mr{S}$ is of characteristic $0$ as above, and let 
 \[
\begin{tikzcd}
\mr{L} \arrow[r, hook, "i"] 
& \mr{X} \arrow[d, "p"]\\
& \mr{S}.
\end{tikzcd}
\]
be a smooth Lagrangian. Let $\omega: \mathscr{O}_\mr{L} \to \Omega_\mr{L/S}[1]$ be the relative Kähler form induced by a relatively ample line bundle and suppose $\mathscr{K} \in \mr{Pic(L)}$ is any (existing) rational power of the canonical bundle of $\mr{L}$. Then the spectral sequence $$\mr{E}_{2}^{p,q} = \mr{R}^{p}p_{*}i_{*}\Omega^{q}_\mr{L/S} \Rightarrow \mr{R}^{p+q}p_{*}\mr{R}\shom(i_{*}\mathscr{K},i_{*}\mathscr{K})$$ degenerates on $\mr{E}_{2}$. 
\end{theorem}
\paragraph{Coisotropic subvarieties.}Another possible generalisation, suggested to us by Sabin Cautis, involves replacing Lagrangian by coisotropic. So let $i : \mr{L \xhookrightarrow{} X}$ be a coisotropic subvariety, then it has a characteristic foliation $p:\mr{L} \to \mr{B}$, which we assume smooth projective, where $\mr{B}$ is the space of leaves. Then $\mathscr{N}_\mr{L/X} \cong \Omega_\mr{L/B}$ and we get a spectral sequence $$\mr{E}_{2}^{p,q} = \mr{R}^{p}p_{*}\Omega^{q}_\mr{L/B} \Rightarrow \mr{R}^{p+q}p_{*}(\mathscr{L}^{\vee}\otimes i^{!}i_{*}\mathscr{L}).$$ One may speculate to what extent it degenerates. Our approach doesn't immediately generalise to this setting, and the reason is that $\mathscr{L}^{\vee}\otimes i^{!}i_{*}\mathscr{L}$ doesn't seem to carry any (natural) algebra structure, so while its cohomology sheaves have cup products, it's not clear if the differentials of the spectral sequence will be compatible with these cup products.

\end{document}